\documentclass[12pt,a4paper]{amsart}
\usepackage[english]{babel}
\usepackage[applemac]{inputenc}
\usepackage{palatino}
\usepackage{amsfonts}
\usepackage{amsmath}
\usepackage{amsthm}
\usepackage{amssymb}
\usepackage{euler}
\usepackage{calrsfs}
\usepackage{graphicx}
\usepackage[colorlinks = true, citecolor = darkblue, linkcolor = red, urlcolor = darkblue, pdfstartview=FitH]{hyperref}
\usepackage[usenames]{color}
\definecolor{darkblue}{rgb}{0,0.08,0.45} 
\newcommand{\R}{\mathbb{R}}

\newcommand{\cD}{{\mathcal{D}}}
\newcommand{\cH}{{\mathcal{H}}}

\renewcommand{\S}{\mathbb{S}}

\newcommand{\x}{\mathbf{x}}
\newcommand{\y}{\mathbf{y}}
\newcommand{\z}{\mathbf{z}}

\renewcommand{\epsilon}{\varepsilon}

\theoremstyle{plain}
\newtheorem{thm}[equation]{Theorem}
\newtheorem{lemma}[equation]{Lemma}

\theoremstyle{definition}

\newtheorem{remark}[equation]{Remark}

\numberwithin{equation}{section}

\author{Tuomas Orponen and Tuomas Sahlsten}\thanks{Both authors are supported by the Finnish Centre of Excellence in Analysis and Dynamics Research and the second author acknowledges the support from Emil Aaltonen S\"{a}\"{a}ti\"{o}.}
\title{Radial projections of rectifiable sets}
\address{Department of Mathematics and Statistics, University of Helsinki, P.O.B. 68, FI-00014 Helsinki, Finland}
\email{tuomas.orponen@helsinki.fi} \email{tuomas.sahlsten@helsinki.fi}

\begin{document}

\begin{abstract} We show that if no $m$-plane contains almost all of an $m$-rectifiable set $E \subset \R^{n}$, then there exists a single $(m - 1)$-plane $V$ such that the radial projection of $E$ has positive $m$-dimensional measure from every point outside $V$.
\end{abstract}

\maketitle 

\section{Introduction}

The purpose of this note is to investigate the relation between rectifiability and \textbf{radial projections} $\pi_{x} \colon \R^{n} \setminus \{x\} \to \S^{n - 1}$, $x \in \R^{n}$, defined by the expression $\pi_{x}(y) = (y - x)/|y - x|$. Motivated by the famous result of Besicovitch and Federer concerning orthogonal projections (see \cite[Theorem 18.1]{Mat95}), one has reason to anticipate that the radial projections $\pi_{x}(E)$ of $m$-rectifiable sets $E \subset \R^{n}$ with $\cH^{m}(E) > 0$ ought to have positive $\cH^{m}$ measure for almost all parameters $x \in \R^{n}$, whereas the opposite behaviour should be manifest in the purely unrectifiable case. The problem of verifying this intuition in the unrectifiable case appears to be rather involved, and only partial results are available (see the discussion and references in section \ref{unrectifiable}); this note complements those results by settling the rectifiable case. We show that if no $m$-plane contains almost all of an $m$-rectifiable set $E \subset \R^{n}$ with $\cH^{m}(E) > 0$, then $\cH^{m}$ almost all radial projections of $E$ have positive $m$-dimensional measure. We also provide an accurate 'worst case' description of the geometry of the \textbf{exceptional set} $\{x \in \R^{n} : \cH^{m}(\pi_{x}(E)) = 0\}$. Our studies were initiated by a question on 'directions of rectifiable sets', raised by Iosevich, Mourgoglou and Senger in the recent preprint \cite{IMS10}: a positive answer to this question is acquired in Remark \ref{ims}.

\section{Preliminaries and the main result}

An $\cH^{m}$ measurable set $E \subset \R^n$ with $\cH^m(E) > 0$ is called $m$-\textbf{rectifiable} if there are countably many $C^1$-embeddings $f_{i} \colon \R^{m} \to \R^{n}$ such that
$$\cH^m\Big(E \setminus \bigcup_{i} f_{i}(\R^{m})\Big) = 0.$$
This definition is equivalent to the classical definition of rectifiability, where we consider coverings by Lipschitz-images instead of $C^1$-manifolds, see \cite[Theorem 15.21]{Mat95}. We include the requirement $\cH^{m}(E) > 0$ in the definition of rectifiability merely to avoid repetition.

We say that a set $E \subset \R^n$ is $m$-\textbf{flat}, if there exists an $m$-plane $T$ such that $E \subset T$. Relaxing this condition slightly, we say that $E$ is \textbf{essentially} $m$-\textbf{flat}, if $\cH^m(E \setminus T) = 0$. Our main result is the following.

\begin{thm} 
\label{thm1}
Let $E \subset \R^{n}$ be an $m$-rectifiable set, which is not essentially $m$-flat. Then $\cH^m(\pi_x(E)) > 0$ for $\cH^m$ almost every $x \in \R^n$.\footnote{Note that $\pi_{x}(E)$ is not, strictly speaking, well defined, if $x \in E$. Here, and in the rest of the paper, the notation $\pi_{x}(E)$ should be interpreted as $\pi_{x}(E \setminus \{x\})$.} Moreover, the exceptional set is always $(m-1)$-flat.
\end{thm}

A $0$-plane is, by definition, a singleton in $\R^n$. Note that the result fails for any essentially flat set $E \subset \R^{n}$. Indeed, if $E$ essentially $m$-flat, then we may find an $m$-plane $T$ such that ${\cH^m(E \setminus T) = 0}$. Write
$$\pi_x(E) = \pi_x(E \setminus T) \cup \pi_x(E \cap T), \qquad x \in \R^{n}.$$ 
The first set on the right is $\cH^m$ null, since $\cH^{m}(E \setminus T) = 0$. For every $x \in T$, the second one is a subset of an $(m-1)$-sphere. Hence $\cH^m(\pi_x(E)) = 0$ for every $x \in T$. 

Furthermore, the exceptional set can be an $(m-1)$-plane. For this just take any two $m$-planes $T_1$ and $T_2$ such that $T_1 \neq T_2$ but $V := T_1 \cap T_2$ is non-empty. Now $V$ is an $(m-1)$-plane. Define $E := T_{1} \cup T_{2}$. Then, for every $x \in V$, the radial projection $\pi_x(E)$ consists of two $(m-1)$-spheres. Hence, $\cH^m(\pi_x(E)) = 0$ for every $x \in V$.

\section{The key lemma}

The main result is a consequence of the following lemma, which concerns the local bilipschitz-properties of the radial projection.

\begin{lemma}
\label{lemma1}
Let $f \colon \R^{m} \to \R^{n}$ be a $C^{1}$-embedding, and let $z_o = f(\x_{o})$ be a point such that $\dim f'(\x_{o})\R^{m} = m$. Then, if $x \notin T := z_{o} + f'(x_{o})\R^{m}$, there exists $\delta > 0$ such that $\pi_x|_{f(\R^{m}) \cap B(z_0,\delta)}$ is bilipschitz.
\end{lemma}

\begin{proof} We may and will assume that $z_{o} = 0$. Write $M := f(\R^{m})$, and fix $y = f(\y), z = f(\z) \in M$. Since $f$ is differentiable at $\z$, say, we have
\begin{align*} y - z & = f(\y) - f(\z) = f'(\z)(\y - \z) + \epsilon(\y - \z)|\y - \z|\\
& = f'(\x_{o})(\y - \z) + [f'(\z) - f'(\x_{o})](\y - \z) + \epsilon(\y - \z)|\y - \z|, \end{align*} 

\noindent where $\epsilon \colon \R^{m} \to \R^{n}$ is, as usual, some function with the property that $\epsilon(h) \to 0$ as $h \to 0$ in $\R^{m}$. Denote the three vectors on the previous line by $t,e_{1}$ and $e_{2}$. Since $\dim f'(\x_{o})\R^{m} = m$, we have the inequality $|t| \geq c|\y - \z|$ for some $c > 0$. On the other hand, choosing $y = f(\y)$ and $z = f(\z)$ close to $z_{o} = 0$ ensures that $\y$ and $\z$ are close to $\x_{o}$, since $f$ is an embedding. It follows that
\begin{displaymath} |e_{1}| \leq \epsilon|\y - \z| \quad \text{and} \quad |e_{2}| \leq \epsilon|\y - \z|, \end{displaymath}

\noindent as soon as $y,z \in B(0,\delta) \cap M$. Above $\epsilon > 0$ can be made arbitrarily small by taking $\delta > 0$ small. Here we also needed the fact that $f$ is continuously differentiable. Denote by $P_{T}$ the orthogonal projection onto $T = f'(\x_{o})\R^{m}$. Since $t \in T$, we may estimate
\begin{align*} |P_{T}(y - z)| & \geq |t| - |P_{T}(e_{1})| - |P_{T}(e_{2})| \geq |t| - |e_{1}| - |e_{2}|\\
& \geq |t| - \epsilon|\y - \z| - \epsilon|\y - \z| \geq |t|(1 - 2\epsilon/c)\\
& = |t|(1 + 2\epsilon/c) \frac{1 - 2\epsilon/c}{1 + 2\epsilon/c} \geq (|t| + \epsilon|\y - \z| + \epsilon|\y - \z|) \frac{1 - 2\epsilon/c}{1 + 2\epsilon/c}\\
& \geq (|t| + |e_{1}| + |e_{2}|)\frac{1 - 2\epsilon/c}{1 + 2\epsilon/c} \geq |y - z|\frac{1 - 2\epsilon/c}{1 + 2\epsilon/c}, \quad y,z \in B(z_{o},\delta) \cap M. \end{align*} 

\noindent The factor of $|y - z|$ tends to one as $\epsilon \to 0$, so this may be re-written more neatly as
\begin{displaymath} |P_{T}(y - z)| \geq (1 - \epsilon)|y - z|, \qquad y,z \in B(0,\delta). \end{displaymath}  

\noindent Next fix $x \notin T$. By making $\delta$ yet smaller if necessary, we may assume that $x - y \notin T$ for any $y \in \bar{B}(0,\delta)$. Hence there exists $\tau < 1$ such that $|P_{T}(x - y)| \leq \tau|x - y|$ for $y \in B(0,\delta)$. Now we claim that there exists $c > 0$ (not necessarily the same as before) such that 
\begin{equation}\label{form1} \gamma(y,z) \geq c|y - z|, \qquad y,z \in B(0,\delta) \cap M, \end{equation}
where $\gamma(y,z)$ is the angle formed by $z - x$ and $y - x$. Assume the contrary and locate $y,z \in B(0,\delta) \cap M$ with $\sin \gamma(y,z) \leq \epsilon|y - z|$. Denote by $w$ the orthogonal projection of $z - x$ onto the line spanned by $y - x$. 
\begin{center}
\includegraphics[scale = 1.3]{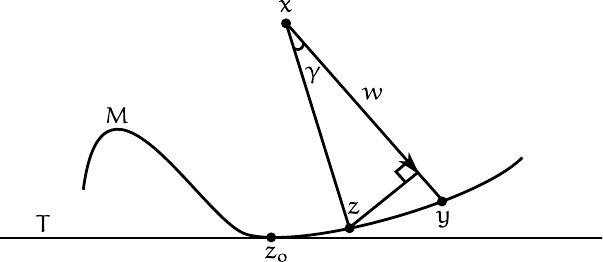}
\end{center}
Then, as the picture shows,
\begin{displaymath} \sin \gamma(y,z) = \frac{|(z - x) - w|}{|x - z|}. \end{displaymath}
Hence $|(z - x) - w| \leq \epsilon|x - z||y - z| \leq C\epsilon|y - z|$. Thus we have shown that $z - x = \lambda(y - x) + v$ for some $\lambda \in \R$ and $|v| \leq C\epsilon|y - z|$. This yields
\begin{displaymath} z - y = x + \lambda(y - x) + v - y = (1 - \lambda)(x - y) + v. \end{displaymath}
Taking $P_{T}$ on both sides and estimating,
\begin{align*} (1 - \epsilon)|y - z| \leq |P_{T}(z - y)| & \leq |(1 - \lambda)||P_{T}(x - y)| + |v|\\
& \leq |(1 - \lambda)| \cdot \tau \cdot |x - y| + |v|\\
& \leq \tau|(1 - \lambda)(x - y) + v| + \tau|v| + |v|\\
& \leq [C \epsilon(1 + \tau) + \tau]|y - z|. \end{align*} 
Since $\epsilon > 0$ could be chosen arbitrarily small and $\tau < 1$, we have reached a contradiction. Thus (\ref{form1}) holds. 

The angle formed by any points $\pi_{x}(y)$ and $\pi_{x}(z)$ on $\S^{n - 1}$ is precisely $\gamma(y,z)$, and this angle is of the same order of magnitude as $|\pi_{x}(y) - \pi_{x}(z)|$. Thus
$$ |\pi_{x}(y) - \pi_{x}(z)| \asymp |\gamma(y,z)| \geq c|y - z|, \qquad y,z \in B(0,\delta) \cap M. $$
Hence $\pi_{x}$ restricted to $B(0,\delta) \cap M$ is bilipschitz. 
\end{proof}

\section{Proof of the main result}

\begin{proof}[Proof of Theorem \ref{thm1}]
Choose a $C^1$-embedding $f_{1} \colon \R^{m} \to \R^{n}$ so that $M_{1} := f_{1}(\R^{m})$ intersects $E$ in a set of positive $\cH^{m}$ measure. Next choose a point $z_1 = f_{1}(\x_{1}) \in E \cap M_1$ such that $\dim f_{1}'(\x_{1})\R^{m} = m$ and $z_{1}$ is a density point of $E \cap M_{1}$. All this is possible by Sard's theorem and the fact that $\cH^{m}$ almost every point of $E \cap M_{1}$ is a density point. Write $T_{1} := z_{1} + f_{1}'(\x_{1})\R^{m}$. By our assumption, $\cH^m(E \setminus T_1) > 0$. Now repeat the previous procedure: choose a $C^1$-embedding $f_{2} \colon \R^{m} \to \R^{n}$ such that $M_{2} := f_{2}(\R^{m})$ intersects $E \setminus T_{1}$ in a set of positive $\cH^{m}$ measure. Also, pick a point $z_2 = f_{2}(\x_{2}) \in (E \setminus T_1) \cap M_2$ such that, again, $\dim f_{2}'(\x_{2})\R^{m} = m$ and $z_{2}$ is a density point of $(E \setminus T_{1}) \cap M_{2}$. Set $T_{2} := z_{2} + f_{2}'(\x_{2})\R^{m}$. Then $T_1 \neq T_2$. 
\begin{center}
\includegraphics[scale = 1]{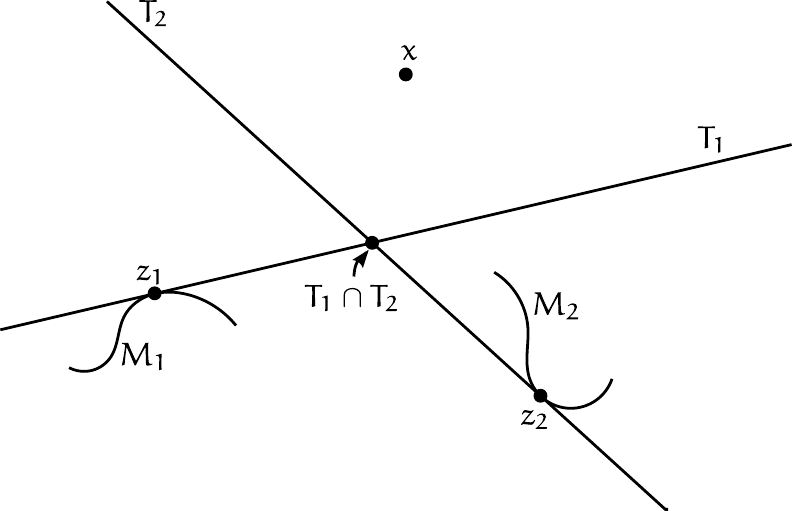}
\end{center}

\noindent Let $x \in \R^n$. If $x \notin T_1$. Then, by Lemma \ref{lemma1}, there exists $\delta > 0$ such that $\pi_x|_{M_{1} \cap B(z_1,\delta)}$ is bilipschitz. Since $E \cap M_1 \cap B(z_1,\delta)$ has positive $\cH^m$ measure, so does $\pi_x(E \cap M_1 \cap B(z_1,\delta))$. Similarly, if $x \notin T_2$, we find $\delta > 0$ such that $\pi_x([E \setminus T_{1}] \cap M_2 \cap B(z_2,\delta))$ has positive $\cH^m$ measure. This implies that the set of all $x \in \R^n$ with $\cH^m(\pi_x(E)) = 0$ is necessarily a subset of $T_1 \cap T_2$, which is always contained in an $(m - 1)$-plane.
\end{proof}

\begin{remark}\label{ims}
In \cite[Conjecture 1.12]{IMS10}, Iosevich \textit{et al.} asked the following: \textit{if $E$ is $(n-1)$-rectifiable and not $(n-1)$-flat, then does the set of directions
$$\cD(E) := \left\{\frac{y-x}{|y-x|} : y,x \in E,\, y \neq x\right\}$$
always have positive $\cH^{n-1}$ measure?} Lemma \ref{lemma1} implies a positive answer: one may even replace $n-1$ by any integer $0 < m < n$. In fact, slightly more is true: if $E \subset \R^{n}$ is an $m$-rectifiable set, which is not $m$-flat, then there exists $x \in E$ such that $\cH^m(\pi_x(E)) > 0$. The proof of this is almost the same as that of Theorem \ref{thm1} -- only simpler. Choose a $C^1$-embedding $f \colon \R^{m} \to \R^{n}$, such that $M := f(\R^{m})$ intersects $E$ in a set of positive $\cH^{m}$ measure. Then choose a point $z_{o} = f(\x_{o}) \in E \cap M$ such that $\dim f'(\x_{o})\R^{m} = m$ and $z_{o}$ is a density point of $E \cap M$. Write $T := z_{o} + f'(\x_{o})\R^{m}$. By our assumption $E \not\subset T$, so we may choose a point $x \in E \setminus T$. Now Lemma \ref{lemma1} yields $\delta > 0$ such that $\pi_x|_{M \cap B(z_o,\delta)}$ is bilipschitz. Since $E \cap M \cap B(z_o,\delta)$ has positive $\cH^m$ measure, so does $\pi_x(E \cap M \cap B(z_o,\delta))$. 
\end{remark}

\section{Radial projections of unrectifiable sets}\label{unrectifiable}

In analogue with the Besicovitch-Federer projection theorem, it seems reasonable to conjecture that the following converse for the result in this note is also true: the radial projections $\pi_x(E)$ of every purely $m$-unrectifiable set $E \subset \R^n$ are $\cH^{m}$ null for $\cH^m$ almost every $x \in E$. This was recently verified for self-similar sets in $\R^{2}$ satisfying the open set condition by Simon and Solomyak in \cite{SiSo06}. In \cite{Mar54, Mat81} Marstrand and Mattila also proved that $\cH^{m}(\pi_{x}(E)) > 0$ can only happen for $x$ in a set of dimension at most $m$. In the plane, at least, this upper bound cannot be improved: Marstrand exhibited an example of a purely $1$-unrectifiable set $E \subset \R^{2}$, which projects radially onto a set of positive length in a set of Hausdorff-dimension exactly $1$. A related phenomenon was discovered by Cs\"{o}rnyei and Preiss in \cite{CsPr07}: given a purely $1$-unrectifiable set $E \subset \R^{2}$ with $\cH^1(E) < \infty$, there exists a $1$-rectifiable set $F \subset \R^{2}$ with $\cH^1(F) < \infty$ such that for $\cH^1$ almost every $x \in E$ almost every line passing through $x$ intersects $F$ in an infinite set.

\end{document}